\documentclass[letterpaper, 10 pt, conference]{ieeeconf}  % Comment this line out if you need a4paper

\IEEEoverridecommandlockouts                              % This command is only needed if 
\usepackage{amsmath,amssymb}
\usepackage{mathtools}
\usepackage{amsthm}
\usepackage{amsfonts}
\usepackage{cite}
\usepackage{comment}
\usepackage{multirow}

\newcommand{\xVar}{\boldsymbol{x}}
\newcommand{\ID}{i}
\newcommand{\dimXVar}{n}
\newcommand{\dimID}{m}
\newcommand{\const}{d}
\newcommand{\constC}{c}
\newcommand{\ineqVar}{\lambda}
\newcommand{\ineqVars}{\boldsymbol{\ineqVar}}
\newcommand{\eqVar}{s}
\newcommand{\KKTVar}{(\ineqVars,\,\eqVar)}
\newcommand{\targetFcnHead}{J}
\newcommand{\targetFcn}{\targetFcnHead(\xVar,\,\ID)}
\newcommand{\realProbHead}{p}
\newcommand{\realProb}{\realProbHead(\ID)}
\newcommand{\probHead}{\hat{p}}
\newcommand{\prob}{\probHead(\ID)}
\newcommand{\refProbHead}{\realProbHead_0}
\newcommand{\refProb}{\refProbHead(\ID)}
\newcommand{\densRatioHead}{r}
\newcommand{\densRatio}{\densRatioHead_\ID}
\newcommand{\inner}{\expectation{\prob}{\targetFcn}}
\newcommand{\innerMax}{\max_{\probHead\in\ambiguity}\quad\inner}

\newcommand{\innerHat}{\max_{\probHead\in\hatAmbiguity}\quad\inner}
\newcommand{\innerHatS}{\max_{\probHead\in\hatAmbiguity}\inner}
\newcommand{\innerDual}{\expectation{\refProb}{\dualExtended}}

\newcommand{\lagrangeanHead}{l_{\xVar}}
\newcommand{\lagrangean}{\lagrangeanHead(\densRatioHead_1,\,\cdots,\,\densRatioHead_\dimID,\ineqVars,\,\eqVar)}
\newcommand{\lagrangeDualHead}{g_{\xVar}}
\newcommand{\lagrangeDual}{\lagrangeDualHead(\ineqVars,\,\eqVar)}
\newcommand{\dualHead}{h}
\newcommand{\dual}{\dualHead(\xVar,\,\ID,\,[\ineqVars]_\ID,\,\eqVar)}
\newcommand{\dualExtended}{\tilde{\dualHead}(\xVar,\,\ID,\,[\ineqVars]_\ID,\,\eqVar)}

\newcommand{\ambiguityKnapsack}{\mathcal{Z}}
\newcommand{\zVar}{\boldsymbol{z}}
\newcommand{\innerKnapsack}{\sum_{\ID\in\setID}\refProb\,\targetFcn\,[\zVar]_\ID}

\newcommand{\innerKnapsackUniform}{ 
\sum_{l=1}^{\constC}\frac{\targetFcnHead(\xVar,\,\ID_l)}{\constC}}
\newcommand{\innerMaxKnapsackUniform}{\max_{(\ID_1,\,\cdots,\,\ID_{\constC})\in\ambiguityC}\quad 
\innerKnapsackUniform}
\newcommand{\ambiguityC}{\ambiguityKnapsack_\constC}

\newcommand{\ambiguityTV}{\ambiguity_{TV}}

\newcommand{\relint}{\mathrm{relint}}
\newcommand{\vect}{\mathrm{vec}}
\newcommand{\eye}{\boldsymbol{I}}
\newcommand{\diag}{\mathrm{diag}}
\newcommand{\lnFcn}{\ln}

\newcommand{\anyVector}{\boldsymbol{v}}
\newcommand{\anyMatrix}{\boldsymbol{C}}
\newcommand{\anySet}{\mathcal{S}}

\newcommand{\anyFcn}{\boldsymbol{f}}
\newcommand{\anyRandom}{s}

\newcommand{\real}{\mathbb{R}}
\newcommand{\realPlus}{[0,\,\infty)}
\newcommand{\realPlusPlus}{(0,\,\infty)}
\newcommand{\setXVar}{\mathcal{X}}
\newcommand{\setID}{{\Omega}}
\newcommand{\ambiguity}{\mathcal{W}}
\newcommand{\hatAmbiguity}{\hat{\mathcal{W}}}
\newcommand{\probabilityHead}{\mathcal{P}}
\newcommand{\probability}{\probabilityHead(\setID)}

\newcommand{\expectation}[2]{\mathrm{E}_{#1}[#2]}
\newcommand{\distribution}{\mathbb{P}}

\theoremstyle{definition}
\newtheorem{dfn}{Definition}

\newtheorem{lem}[dfn]{Lemma}
\newtheorem{thm}[dfn]{Theorem}

\newtheorem{rem}[dfn]{Remark}

\title{\LARGE \bf Discrete Distributionally Robust Optimal Control with Explicitly Constrained Optimization}

\author{Yuma Shida$^{1}$ and Yuji Ito$^{1}$% <-this % stops a space
\thanks{*This work has been submitted to the IEEE for possible publication. Copyright may be transferred without notice, after which this version may no longer be accessible. This work was not supported by any organization.}% <-this % stops a space
\thanks{$^{1}$Toyota Central R\&D Labs., Inc., 41-1, Yokomichi, Nagakute, Aichi, 480-1192, Japan.
        {\tt\small \{Yuma.shida.fw, ito-yuji\}@mosk.tytlabs.co.jp}%; {\tt\small yuma\_shida@mail.toyota.co.jp}
}%
}

\begin{document}

\maketitle
\thispagestyle{empty}
\pagestyle{empty}

%%%%%%%%%%%%%%%%%%%%%%%%%%%%%%%%%%%%%%%%%%%%%%%%%%%%%%%%%%%%%%%%%%%%%%%%%%%%%%%%
\begin{abstract}

Distributionally robust optimal control (DROC) is gaining interest. This study presents a reformulation method for discrete DROC (DDROC) problems to design optimal control policies under a worst-case distributional uncertainty. The reformulation of DDROC problems impacts both the utility of tractable improvements in continuous DROC problems and the inherent discretization modeling of DROC problems. DROC is believed to have tractability issues; namely, infinite inequalities emerge over the distribution space. Therefore, investigating tractable reformulation methods for these DROC problems is crucial. One such method utilizes the strong dualities of the worst-case expectations. However, previous studies demonstrated that certain non-trivial inequalities remain after the reformulation. To enhance the tractability of DDROC, the proposed method reformulates DDROC problems into one-layer smooth convex programming with only a few trivial inequalities. The proposed method is applied to a DDROC version of a patrol-agent design problem.

\end{abstract}

%%%%%%%%%%%%%%%%%%%%%%%%%%%%%%%%%%%%%%%%%%%%%%%%%%%%%%%%%%%%%%%%%%%%%%%%%%%%%%%%
\section{Introduction}

Real-world systems involve several uncertainties. For example, it is difficult for security robots\cite{duan2021markov,diaz2023distributed} to predict the location of a noteworthy event. A popular approach for controlling these uncertain systems is to minimize the control costs with a stochastic model of uncertainty: stochastic optimal control (SOC)\cite{bertsekas1996stochastic,crespo2003stochastic}. If a stochastic model is expensive and/or difficult to obtain, robust control (RC)\cite{sastry2011adaptive,slotine1985robust,scherer2001theory} can adress a wide range of uncertainties by considering the worst-case value of the control costs. However, using RC carries the risk of conservative control results.

Distributionally robust optimal control (DROC) has been recently developed to overcome the conventional challenges of both SOC and RC. 
%: the lack of precise stochastic models and conservative design. 
DROC techniques can enhance the robustness against distributional uncertainties arising from mismatches between real-world systems and stochastic models, such as the well-known Gaussian noise assumptions \cite{nishimura2021rat}. DROC aims to minimize the expectation of the cost function using a worst-case distribution \cite{taskesen2024distributionally,liu2023data,yang2020wasserstein,nishimura2021rat,nguyen2023distributionally,coulson2021distributionally}. Distributionally robust constraints regarding the value at risk have also been considered \cite{van2015distributionally,pilipovsky2024distributionally}. These distributional uncertainties are known as ambiguity sets, which are expressed by bounds or balls of statistical measures between probability distributions, such as $\phi$-divergence \cite{liu2023data,nishimura2021rat,hu2013kullback} and optimal transport distance \cite{taskesen2024distributionally,yang2020wasserstein,nguyen2023distributionally,pilipovsky2024distributionally,coulson2021distributionally,gao2023distributionally,shafieezadeh2019regularization,cherukuri2022data,luo2017decomposition,mohajerin2018data}. 

%and optimal transport distance \cite{taskesen2024distributionally,yang2020wasserstein,nguyen2023distributionally,pilipovsky2024distributionally,coulson2021distributionally,gao2023distributionally,shafieezadeh2019regularization,cherukuri2022data,luo2017decomposition,mohajerin2018data}. Especially in data-driven approaches of the DROC and distributionally robust optimization (DRO), it is thought that the $p$-wasserstein distance has much advantages \cite{mohajerin2018data}, rather than $\phi$-divergence.

The solvability of the DROC problems is a critical problem. Specifically, semi-infinite programming (SIP), which involves infinite inequalities\cite{reemtsen2013semi},  generally emerges in DROC problems \cite{yang2020wasserstein,shafieezadeh2019regularization,cherukuri2022data,mehrotra2014cutting,luo2017decomposition,mohajerin2018data,wiesemann2014distributionally}. These infinite inequalities remain computational challenges must be solved. There are two approaches for solving the SIP: 1) directly solving the SIP using several algorithms such as the cutting-surface method\cite{shafieezadeh2019regularization,cherukuri2022data,mehrotra2014cutting,luo2017decomposition,mohajerin2018data}, and 2) reformulating the SIP into finite convex programming via a strong duality, % \cite{mohajerin2018data,gao2023distributionally,yang2020wasserstein, shafieezadeh2019regularization,miao2021data}, 
such as the Kantorovich duality\cite{mohajerin2018data,gao2023distributionally,yang2020wasserstein}, duality of conic linear programming\cite{shafieezadeh2019regularization}, and duality of convex optimization\cite{miao2021data}. Other previous studies have reformulated SIP into semi-definite programming \cite{pilipovsky2024distributionally,van2015distributionally,staib2019distributionally,gao2023distributionally,wiesemann2014distributionally}. Meanwhile, the existing studies\cite{yang2020wasserstein,liu2023data,taskesen2024distributionally} have avoided SIP by choosing specific problem settings, such as linear-quadratic settings.
%In the meanwhile, there exist some cases where it is possible to avoid the SIP by choosing specific problem settings. The existing work \cite{yang2020wasserstein} has chosen linear quadratic control objects and the $2$-wasserstein distance, and \cite{nguyen2023distributionally} has chosen model predictive control with convex objective funcitons and random variables supported on convex sets that have vertexes. The existing work \cite{taskesen2024distributionally} has choosen a linear quadratic gaussian control objects and the Gelbrich distance, and find a Nash equilibrium of a zero-sum game as a solution to the DROC problem.

Discrete DROC (DDROC) affects both tractable discretization \cite{liu2019discrete} and inherent discrete modeling \cite{farokhi2023distributionally} for DROC. The discretization of probability distributions is believed to play a role in improving the tractability of distributionally robust optimization (DRO) problems where limited empirical data are available \cite{liu2019discrete}. Furthermore, this DDROC and discrete DRO (DDRO) also emerge when considering distributions of discrete modeling that haves finite space\cite{farokhi2023distributionally,zhang2021efficient}; in fact, robotic surveillance studies have used discrete modeling of finite locations\cite{duan2021markov,diaz2023distributed}. Compared to continuous DROC, discrete distributions are suitable for representing inherently multi-modal distributions. 
In this study, we establish more tractable DDROC problems than those reported in previous studies \cite{farokhi2023distributionally,zhang2021efficient}. Our proposed method reformulates min-max problems encountered in solving the DDROC into one-layer smooth convex programming with only a few trivial constraints. Specifically, the constraints reduce to non-negativeness of the dual minimizers (Lagrange multipliers). In previous studies, reformulated DDROC and DDRO problems either contained certain non-trivial inequalities or were not one-layer smooth convex programming problems. 

Our contributions include solvability, explainability, and demonstration. 
First, regarding solvability, our setting of distributional uncertainty, called the ambiguity set, realizes the aforementioned reformulation to yield one-layer smooth
convex programming with trivial constraints. The ambiguity set is defined using a density ratio between a nominal probability density and a density. 
%regarding solvability, we consider DDROC problems with an ambiguity set of probability density that is a smooth subset of the total variation (TV) ball, where the TV distance is an optimal transport distance\cite{villani2009optimal}. Our proposed method can reformulate problems without non-trivial inequalities using a strong duality of convex programming. The proposed method is one-layer minimized-minimized smooth convex programming. 
Second, in terms of explainability, the problems reformulated by our proposed method can be explained as deterministic RC problems, in which a maximizer corresponds to a collection of discrete variables. Finally, we demonstrate that our reformulated problems can be solved by general convex programming through numerical experiments on the patroller-agent design of \cite{diaz2023distributed}. This design was originally proposed as a minimization problem of either a weighted average value or the worst-case value of a mean hitting time, instead, we adapt it to fit the framework of DDROC problems. 
%\begin{itemize} \item Solvability: We consider the DDROC problems with an ambiguity set of a probability density that is a smooth subset of the TV ball. Our proposed method can reformulate the problems into problems without non-trivial inequalities, by using strong duality of convex programming. The proposed method becomes the one-layer minimize-minimize smooth convex programming. \item Explainability: The problems that our proposed method reformulates into can be explained as deterministic RC problems to which a maximizer corresponds to a collection of a discrete variable. \item Demonstration: We demonstrate that our reformulated problems can be solved by general convex programming throughout numerical experiments of patroller agent design of \cite{diaz2023distributed}, originally supposed as a minimization problem of either a weighted average value or the worst-case value of a mean hitting time. In contrast, we considered that as the DDROC problems. \end{itemize}

\section*{Notation}
The following notations are used in this study:
\begin{itemize}
\item $\eye_a$: $a\times a$ identity matrix 

%\item $\zeros_{a\times b}$: the $a\times b$ zero matrix

\item $[\anyVector]_j$: component in the $j$-th element of a vector $\anyVector\in\real^a$

\item $[\anyMatrix]_{j,\,k}$: component in the $j$-th row and $k$-th column of a matrix $\anyMatrix\in\real^{a\times b}$

\item $\diag(\anyVector) \coloneq 
\begin{bmatrix}
[\anyVector]_{1} & & 0 \\
 & \ddots & \\
0 & & [\anyVector]_{a}
\end{bmatrix}
$: diagonal matrix of the components of a vector $\anyVector\in\real^{a\times a}$

\item $\vect(\anyMatrix) \coloneq
[
[\anyMatrix]_{1,\,1}  \cdots  [\anyMatrix]_{a,\,1}  \cdots  [\anyMatrix]_{1,\,b}  \cdots  [\anyMatrix]_{a,\,b} 
]
^\top$: vectorization of the components of a matrix $\anyMatrix\in\real^{a\times b}$ 

%\item $\dom(\anyFcn)$: a subset of $\real^a$ for which $\anyFcn:\real^a\rightarrow\real^b$ is defined, e.g., $\dom(\ln)=\realPlusPlus$

\item $\relint(\anySet)$: relative interior of a set $\anySet\subseteq\real^a$

%\item $\cl(\anySet)$: the closure of a set $\anySet\subseteq\real^a$

%\item $\anyMatrix\succ0$: the positive definiteness of a symmetrix matrix $\anyMatrix=\anyMatrix^\top\in\real^{a\times a}$

%\item $\anyMatrix\succeq0$: the positive semi-definiteness of a symmetrix matrix $\anyMatrix=\anyMatrix^\top\in\real^{a\times a}$

%\item $\derivativeD{\anyFcn(\anyVector)}{\anyVector}$: the partial derivative of $\anyFcn(\anyVector)$ in $\anyVector$; the component in the $j$-th row and the $k$-th column of that is $[\derivativeD{\anyFcn(\anyVector)}{\anyVector}]_{j,k}=\derivativeScalerD{[\anyFcn(\anyVector)]_j}{[\anyVector]_k}$

%\item $\hessian{\anyFcnHead}\coloneq\derivativeDTwice{\anyFcnHead(\anyVector)}{\anyVector}$: the hessian matrix of $\anyFcnHead(\anyVector)$ in $\anyVector$; the component in the $j$-th row and the $k$-th column of that is $[\hessian{\anyFcnHead}]_{j,k}=[\derivativeDTwice{\anyFcnHead((\anyVector))}{\anyVector}]_{j,k}=\dfrac{\partial^2 \anyFcnHead(\anyVector)}{\partial [\anyVector]_j \partial [\anyVector]_k}$

\item $\displaystyle\probabilityHead(\anySet) \coloneq \{\realProbHead:\anySet\rightarrow\realPlus~|~\sum_{\anyRandom\in\anySet}\realProbHead(\anyRandom)=1\}$: set of all the probability mass functions of a random discrete variable $\anyRandom\in\anySet$ on a finite set $\anySet$ 

%\item $\displaystyle\probabilityHead(\anySet) \coloneq \{\realProbHead:\anySet\rightarrow\realPlusPlus|\sum_{\anyRandom\in\anySet}\realProbHead(\anyRandom)=1\}$: set of all probability mass function of a random discrete variable $\anyRandom\in\anySet$ supported on a discrete set $\anySet$ 

\item $\expectation{\realProbHead(\anyRandom)}{\anyFcn(\anyRandom)}$: expectation of $\anyFcn(\anyRandom)$ with respect to a random variable $\anyRandom$ which has a probability density $\realProbHead(\anyRandom)$

\item $\distribution[\anyRandom\in\anySet]$: probability that a random variable $\anyRandom$ belongs to an event $\anySet$

\item $\distribution[\anyRandom_1\in\anySet_1~|~\anyRandom_2\in\anySet_2]$: conditional  probability that a random variable $\anyRandom_1$ belongs to an event $\anySet_1$ under the condition that a random variable $\anyRandom_2$ belongs to an event $\anySet_2$ 
\end{itemize}

\section{Target Systems and Problems Setting}

\subsection{Target Systems}
We consider a target system that has a decision variable $\xVar\in\setXVar\subseteq\real^\dimXVar$ on a set $\setXVar$ and a random variable $\ID\in\setID$ with a probability density $\realProbHead\in\probability$ on a discrete finite set $\setID=\{1,\,2,\,\cdots,\,m\}$. This study assumes that the probability density $\realProb$ is unknown.   
The performance of the target system is indicated by expectation $\expectation{\realProb}{\targetFcn}$ of the cost function $\targetFcn$. The cost function $\targetFcnHead:\real^\dimXVar\times\setID\to\mathbb{R}$ represents the control objective to be minimized. For example, in robotic control, this may be the time required for the robot to reach its target location. The definition of $\targetFcn$ indicates that for each $\xVar$, $\expectation{\realProb}{\targetFcn} < \infty$.

%We consider a target system that has a decision variable $\xVar\in\setXVar\subseteq\real^\dimXVar$ on a  set $\setXVar$ and a random variable $\ID\in\setID$ with a probability density $\realProbHead\in\probability$ supported on a discrete finite set $\setID=\{1,2,\cdots,m\}$. The real probability $\realProbHead(\ID^*)=\distribution[\ID=\ID^*]$ of a random variable $\ID$ for each $\ID^*\in\setID$ is not known. The performance of the target system is indicated by the expectation $\expectation{\realProb}{\targetFcn}$ of a cost function $\targetFcnHead:\real^\dimXVar\times\setID\to\mathbb{R}$. 

%\begin{ass}[Finite Output of Objective Function]\label{ass:finiteFcn}
%Let us assume that the cost function $\targetFcn$ has a finite output:
%\begin{equation}
%\forall\xVar\in\setXVar, \quad \forall\ID\in\setID, \quad\targetFcn<\infty.
%\end{equation}
%Hence, we can immediately derive the following relationship:  

%\begin{equation}
%\forall\xVar\in\setXVar, \quad \forall\realProbHead\in\probability, \quad\expectation{\realProb}{\targetFcn}<\infty.
%\end{equation}
%\end{ass}

\subsection{Discrete Distributionally Robust Optimal Control Problems}

We consider the DDROC problems associated with the target system and the following ambiguity set, assuming that it contains the probability density $\realProb$ as follows:

\begin{equation}\label{eq:ambiguity}
\ambiguity \coloneqq \{\probHead\in\probability~|~\forall \ID\in\setID, ~ \densRatio\le1+\const\}.
\end{equation}
Here, $\densRatio\coloneq\prob/\,\refProb\in\realPlus$ is the probability density ratio between a nominal probability density $\refProbHead\in\probability$ and any probability density $\probHead\in\probability$. Let us suppose that $\refProb>0$ is satisfied for each $\ID\in\setID$. A positive constant $\const>0$ controls the size of the ambiguity set in (\ref{eq:ambiguity}).

\textit{DDROC problem:} A decision variable is designed to minimize the worst-case expectation of the cost function of the target system in the ambiguity set:

\begin{equation}\label{eq:DROC}
\min_{\xVar\in\setXVar}\quad\innerMax.
\end{equation}

%\begin{rem}[Robust Counterpart of Discrete SOC]
%From a defination of maximize, we can immediately derive the inner worst expectation of (\ref{eq:DROC}) as a robust counterpart of the nominal expectation, so as:

%\begin{multline}\label{eq:counterpart}
%\forall\xVar\in\setXVar,\forall\realProbHead\in\ambiguity, \\\quad\innerMax\ge\expectation{\realProb}{\targetFcn}.
%\end{multline}
%\end{rem}

\begin{rem}[Differentiable Subset of the Optimal Transport Ball]
The density ratio is useful because it is smooth (differentiable), unlike the total variation (TV) distance, which is equal to half of $L_1$ distance. Consider the following TV ball:

\begin{equation}\label{eq:ambiguityTV}
\ambiguityTV \coloneq \{\probHead\in\probability~|~\expectation{\refProb}{|\densRatio-1|}\le\const\},
\end{equation}
which is a special case of optimal transport balls \cite{villani2009optimal}. 
Based on the density-ratio ball $\ambiguity$ denoted in (\ref{eq:ambiguity}), any probability density $\probHead\in\ambiguity$ satisfies $\probHead\in\ambiguityTV$ if $\const\ge1$. 
Different types of relationships between the density ratio and the TV distance or other measures can be observed in \cite[Proposition 1]{10.3150/20-BEJ1256}. 
These facts imply that the density-ratio ball $\ambiguity$ is a subset of the TV ball $\ambiguityTV$ in (\ref{eq:ambiguityTV}). 
\end{rem}

\section{Proposed Method}

To solve the aforementioned DDROC problem, we reformulate the min-max problem (\ref{eq:DROC}) into a tractable form in Section \ref{sec:mainResult}. Specifically, Theorem \ref{thm:strongDual} shows that (\ref{eq:DROC}) reduces to a minimization problem with only trivial inequalities. Theorem \ref{thm:knapsack} implies that (\ref{eq:DROC}) can be interpreted as a deterministic RC problem, which provides an intuition of the problem. Section \ref{sec:subResult} describes the proofs of these theorems.
%In this section, we introduce our proposed method, by reformulating DDROC problem (\ref{eq:DROC}) into a convex problem without inequalities. Firstly, in section \ref{sec:mainResult}, we introduce a weak dual problem with small duality gap and a strong dual problem with no guality gap, associated with the DDROC problem $\problem$ in (\ref{eq:DROC}). Subsequently, in following sections after \ref{sec:subResult}, we derive progress of the reformulation and clarify criterias the reformulated problem becom convex optimization.

\subsection{Main Results: Reformulation of Discrete Distributionally Robust Control Problems}\label{sec:mainResult}
We introduce the Lagrange multipliers $\ineqVars\in\realPlus^\dimID$ and $\eqVar\in\real$.  We also introduce the following problem:

\begin{equation}\label{eq:dualDROC}
\min_{\xVar\in\setXVar}\quad\inf_{(\ineqVars,\eqVar)\in\realPlus^\dimID\times\real}\innerDual.
\end{equation}
Here, ${\tilde{\dualHead}} :\setXVar\times\setID\times\real\times\real\rightarrow\real\cup\{\infty\}$ is denoted as the following functions:

\begin{equation}\label{eq:zeroWorst}
\dualExtended \coloneqq
\begin{cases}
\eqVar, & \begin{aligned}(&[\ineqVars]_\ID=0, ~ \\ &\targetFcn\le\eqVar),\end{aligned} \\
\dual, & ([\ineqVars]_\ID\ne0),\\
\infty, & \begin{aligned}(&[\ineqVars]_\ID=0, ~ \\ &\targetFcn>\eqVar),\end{aligned}
\end{cases}
\end{equation}

\begin{equation}\label{eq:worst}
\dual \coloneqq (1+\const)\,[\ineqVars]_\ID\exp(\frac{\targetFcn-[\ineqVars]_\ID-\eqVar}{[\ineqVars]_\ID})+\eqVar.
\end{equation}

\begin{thm}[Tractable Discrete Distributionally Robust Control Problems]\label{thm:strongDual}
The problem in (\ref{eq:dualDROC}) satisfies the following properties.
\begin{itemize}
\item Optimal minimizers of $\xVar$ to (\ref{eq:dualDROC}) are equivalent to those to the DDROC problem in (\ref{eq:DROC}):

\item If the cost function $\targetFcn$ is strictly convex and continuous on a bounded closed convex set $\setXVar$ for each $\ID\in\setID$, the optimal minimizer to (\ref{eq:dualDROC}) is unique.

\item If the cost function $\targetFcn$ is convex and class $C^k$ on a open convex set $\setXVar$ for each $\ID\in\setID$, the objective function of (\ref{eq:dualDROC}) is also convex and class $C^k$ on $\setXVar\times\realPlusPlus^\dimID\times\real$.
\end{itemize}
\end{thm}

\begin{rem}[Solvability of the DDROC Problems]
The problem in (\ref{eq:dualDROC}) is derived from the DDROC problem in (\ref{eq:DROC}). It is a one-layer minimization (infimum) problem with only trivial inequalities, $[\ineqVars]_\ID\ge0$ for all $\ID\in\setID$. 
Futhermore, if the cost function $\targetFcn$ is convex and sufficiently smooth (continuously differentiable) on a convex set $\setXVar$ for each $\ID\in\setID$, the problem in (\ref{eq:dualDROC}) becomes a one-layer smooth convex programming that can be solved by general gradient-based algorithms such as the interior point method\cite{boyd2004convex}. 
%However, the properties $[\ineqVars]_\ID=0$ for any $\ID\in\setID$ may lead a risk of stopping  convergence to the optimal solution. In order to avoid that, let us reformulate the problem into the more tractable problem as follows:

%\begin{multline}\label{eq:weakDualDROC}
%\displaystyle\min_{\xVar\in\setXVar}\barrierInnerMin.
%\end{multline}
%Here, $\barrierConst>0$ is a positive constant and $\barrierHead:\real\rightarrow\real\cup\{-\infty\}$ is logarithmic barrier function to avoid $[\ineqVars]_\ID=0$ as follows: 

%\begin{equation}\label{eq:barrier}
%\barrierFcn=
%\begin{cases}
%\lnFcn([\ineqVars]_\ID), & [\ineqVars]_\ID\in\realPlusPlus,\\ -\infty, & [\ineqVars]_\ID\notin\realPlusPlus.
%\end{cases}
%\end{equation}  
\end{rem}

Subsequently, we present another reformulation of (\ref{eq:DROC}) to interpret the physical meaning of the ambiguity set. We introduce a deterministic variable $\zVar\in\real^\dimID$ and a set of the variable $\zVar$; each component of the variable with a weight $(1+\const)\,\refProb$ is constrained as follows: 

\begin{equation}\label{eq:knapsack}
\ambiguityKnapsack 
\coloneq \{\zVar\in[0,\,1]^\dimID~|~(1+\const)\sum_{\ID\in\setID}\refProb\,[\zVar]_\ID=1 \}. 
\end{equation}
We also introduce the deterministic RC problem in which the total of each $[\zVar]_\ID$ with the weighted cost $\refProb\,(1+\const)\,\targetFcn$ is maximized in (\ref{eq:knapsack}), as follows:

\begin{equation}\label{eq:knapsackRC}
\min_{\xVar\in\setXVar}\quad\max_{\zVar\in\ambiguityKnapsack}\quad \innerKnapsack.
\end{equation}

\begin{thm}[Deterministic RC Problems with Weights]\label{thm:knapsack}
The problem in (\ref{eq:DROC}) satisfies the following properties:
\begin{itemize}
\item Optimal minimizers of $\xVar$ to (\ref{eq:DROC}) are equivalent to those to the RC problem in (\ref{eq:knapsackRC}).
\item If the nominal probability is uniform, that is, $\refProb=1/\,\dimID$, optimal minimizers to (\ref{eq:DROC}) are equivalent to those to the following RC problem: 

\begin{equation}\label{eq:knapsackRC2}
\min_{\xVar\in\setXVar}\quad\innerMaxKnapsackUniform,
\end{equation}
\begin{equation}
\ambiguityC\coloneq\{(\ID_1,\,\cdots,\,\ID_\constC)\in\setID^\constC~|~\forall(j,\,k)\in\setID\times\setID,~\ID_j\ne\ID_k\},
\end{equation}
%that minimizing average value of any cost function collection of size $\constC\coloneq\dimID/\,(1+\const)$, 
for $\constC\coloneq\dimID/\,(1+\const)$, provided that $\constC$ is an integer. % in $\{1,\,\cdots,\,\dimID\}$.
\end{itemize}
\end{thm}

\begin{rem}[Size Explanation of the Ambiguity Set]
The problem in (\ref{eq:knapsackRC2}), which is the minimization of the total of the worst $\constC$ collection of costs, can be used to explain the size of ambiguity set $\ambiguity$. If $\constC$ nears one, the size of ambiguity set, $\const$, increases. In addition, the problem nears minimization of the worst-case cost. While $\constC$ increases, $\const$ decreases and the problem nears minimization of the average value of the costs.
%If we set size $\const$ such that there exists $c\in\{1,\cdots,\dimID\}$ that satisfies $c=\dimID/\,(1+\const)$,   
\end{rem}

\subsection{Proofs of Theorems \ref{thm:strongDual} and \ref{thm:knapsack}}\label{sec:subResult}

%\begin{rem}[How to Prove of the Theorems \ref{thm:strongDual} and \ref{thm:knapsack}]
We prove Theorems \ref{thm:strongDual} and \ref{thm:knapsack} after deriving the following Lemmas \ref{lem:ambiguity} and \ref{lem:lagrangeDual}.
%\end{rem}

Set $\hatAmbiguity$ is denoted as follows: 

\begin{equation}\label{eq:hatAmbiguity}
\hatAmbiguity \coloneq \{\probHead\in\probability~|~\forall \ID\in\setID, ~ \densRatio\lnFcn(\densRatio)\le\densRatio\lnFcn(1+\const)\}, 
\end{equation}
%\begin{rem}[The Constraint Can Be Defined in Zero]
where this study defines $0\lnFcn(0)=0$. This is owing to the continuity as $\lim_{\densRatio\to+0}\densRatio\lnFcn(\densRatio)=0$ \cite[Section 2.1]{cover1999elements}. 
%\end{rem}

\begin{lem}[Equality of the Ambiguity Set]\label{lem:ambiguity}
Set $\hatAmbiguity$ in (\ref{eq:hatAmbiguity}) is equivalent to the ambiguity set in (\ref{eq:ambiguity}) as follows:

\begin{equation}\label{eq:ambiguityReformulation}
\hatAmbiguity=\ambiguity. %, \quad \text{i.e.}, \quad \probHead\in\hatAmbiguity\Leftrightarrow\probHead\in\ambiguity .
\end{equation}
\end{lem}

\begin{proof}[Proof of Lemma \ref{lem:ambiguity}]
Clearly, $\densRatio\lnFcn(\densRatio)\le\densRatio\lnFcn(1+\const)$ if $\densRatio=0$; then, we assume $\densRatio>0$. We have $\lnFcn(\densRatio)\le\lnFcn(1+\const)\Leftrightarrow \densRatio\le1+\const$ because $\ln$ is monotonically increasing. In addition, by multiplying $\densRatio>0$ to the left side, the statement is proven. 
\end{proof}

\begin{lem}[Strong Duality of the Worst Expectation]\label{lem:lagrangeDual}
The following properties are satisfied:

\begin{itemize}
\item For every $\xVar\in\setXVar$, we have:

\begin{multline}\label{eq:inner}
\innerMax = \\
\inf_{\KKTVar\in\realPlus^\dimID\times\real}\innerDual.
\end{multline}

\item If the cost function $\targetFcn$ is a convex function on $\setXVar$ for each $\ID\in\setID$, the objective function on the right side of (\ref{eq:inner}) is convex on $\setXVar\times\realPlus^\dimID\times\real$.
\end{itemize}

%Additionally, if there exists a pair $(\ineqVars^*,\eqVar^*)$ with $\ineqVars^*>0$ as a optimal variables of the right side minimizer in (\ref{eq:inner}), ${\densRatioHead}_{\xVar,\ID}^* :\realPlusPlus^\dimID\times\real\times\setXVar\rightarrow(0,\infty)$ of (\ref{eq:worst}) is a element of the left side maximizer in (\ref{eq:inner}), so as:

%\begin{equation}
%\refProb\worstDensRatio\in \argmax_{\probHead\in\ambiguity}\inner
%\end{equation}
\end{lem}

\begin{rem}[Proof Ideas of Lemma \ref{lem:lagrangeDual}]
The following proof of Lemma \ref{lem:lagrangeDual} is based on ideas from previous studies\cite{liu2023data,nishimura2021rat,hu2013kullback} to explicitly derive a Lagrange dual function. %the derivation of entropic risk measure from DROC and DRO penalized via Kullback-Leibler divergence.
\end{rem}

\begin{proof}[Proof of Lemma \ref{lem:lagrangeDual}]
The following equation is obtained from Lemma \ref{lem:ambiguity}:

\begin{equation*}
\innerHat=\innerMax.
\end{equation*}
In addition, a strong dual problem arises from Slater's condition. At that time, Slater's condition consists of the existence of a point which is called strictly feasible, and the problem before reformulating must be convex \cite{boyd2004convex}. %\begin{rem}[Convexity and Strict Feasibility]\label{rem:ambiguity}
The objective function $\inner$ is linear (concave) in $\prob$ for each $\ID\in\setID$, and $\ambiguity=\hatAmbiguity$ is a convex set. Hence, $\innerHatS$ is convex programming (linear programming). 
Furthermore, $\relint(\hatAmbiguity)$ is a non-empty set if $\const>0$. Hence, $\probHead\in\relint(\hatAmbiguity)$ exists such that it is strictly feasible; namely, satisfies $\densRatio<d+1$ for all $\ID\in\setID$. %Hence, the worst expectation $\innerMaxS$ satisfies Slater's condition\cite{boyd2004convex}. These results leads that we can derive dual problem of the worst expectation that is strong dual. \end{rem}

The dual problem is denoted  as follows: 

\begin{equation*}
\inf_{(\ineqVars,\eqVar)\in\realPlus^\dimID\times\real}\lagrangeDual=\innerHat.
\end{equation*}
Here, let us introduce the Lagrange dual function \cite{boyd2004convex}, $\lagrangeDualHead:\setXVar\times\real^\dimID\times\real \rightarrow \real\cup\{\infty\}$ is associated with the problem in (\ref{eq:inner}) with $\hatAmbiguity$ as follows:

\begin{equation*}%\label{eq:lagrangeDual}
%\lagrangeDual=\displaystyle\sup_{(\densRatioHead_1,\densRatioHead_2,\cdots\densRatioHead_\dimID)\in\realPlusPlus^\dimID}\quad\lagrangean,
\lagrangeDual=\displaystyle\sup_{(\densRatioHead_1,\,\densRatioHead_2,\,\cdots,\,\densRatioHead_\dimID)\in\realPlus^\dimID}\quad\lagrangean,
\end{equation*}

\begin{multline*}%\label{eq:lagrangean}
\lagrangean
=\expectation{\refProb}{\densRatio\targetFcn} \\
+\expectation{\refProb}{\densRatio[\ineqVars]_\ID\lnFcn(\frac{1+\const}{\densRatio})}
+\eqVar(1-\expectation{\refProb}{\densRatio}).
\end{multline*}
The Lagrangean \cite{boyd2004convex}, $\lagrangeanHead: \setXVar\times\realPlus^\dimID\times\real^\dimID\times\real \rightarrow \real$, is associated with the problem in (\ref{eq:inner}) with $\hatAmbiguity$. 

The first statement of the lemma is proven by explicitly deriving the Lagrange dual function $\lagrangeDual$. First, we consider the case $\ID\in\{l\in\setID~|~[\ineqVars]_l>0\}$. The Lagrangean is concave in $\densRatio$; thus, the gradient of that in $\densRatio$ must be zero at the optimal maximizer $\densRatio^*$ as follows:

\begin{equation*}%\label{eq:gradientLagrangean}
%\forall \ID\in\{l\in\setID|[\ineqVars]_l>0\}, \quad 
%\derivativeScaler{\lagrangean}{\densRatio}|_{\densRatio=\densRatio^*} \\ =
\refProb\,\targetFcn 
+\refProb\,[\ineqVars]_\ID\{\lnFcn(\frac{1+\const}{\densRatio^*})
-1\}-\refProb\eqVar
=0.
\end{equation*}
Therefore, 

\begin{equation*}
\densRatio^*=(1+\const)\exp(\frac{\targetFcn-[\ineqVars]_\ID-\eqVar}{[\ineqVars]_\ID}). 
\end{equation*}
Hence, the terms in the Lagrangean are as follows:

\begin{multline*}
%\begin{split}
%\forall \ID\in\{l\in\setID|[\ineqVars]_l>0\}, \quad \\ 
\refProb\densRatio^*\targetFcn
+
\refProb\densRatio^*[\ineqVars]_\ID\lnFcn(\frac{1+\const}{\densRatio^*}) 
+\refProb\eqVar(1-\densRatio^*) \\
=\refProb\dual
.
%\end{split}
\end{multline*}
Subsequently, let us consider the case $\ID\in\{l\in\setID~|~[\ineqVars]_l=0\}$. Then, the Lagrangean $\lagrangean$ is affine in $\densRatioHead_\ID$ on $\realPlus$ as follows:

\begin{equation*}
\begin{matrix}
\refProb\densRatio\targetFcn+\\
\quad\quad\quad
\refProb\eqVar(1-\densRatio)\in
\end{matrix}
\begin{cases}
[\refProb\eqVar,\,\infty), & \text{($\eqVar<\targetFcn$)}, \\
\{\refProb\eqVar\}, & \text{($\eqVar=\targetFcn$)}, \\
(-\infty,\,\refProb\eqVar], & \text{($\eqVar>\targetFcn$)}. \\
\end{cases}
\end{equation*}
Lastly, considering both cases, we can explicitly denote it as follows:

\begin{equation*}
\forall\ineqVars\in\realPlus^\dimID, \quad \lagrangeDual = \innerDual.
\end{equation*}
Hence, the first statement is proven.

We also prove the second statement of the lemma using the last equation. For each $(\densRatioHead_1,\,\cdots,\,\densRatioHead_\dimID)$, the Lagrangean $\lagrangean$ is convex on $\setXVar\times\realPlus^\dimID\times\real$ if $\targetFcn$ is convex on $\setXVar$ for each $\ID\in\setID$. Hence, based on the result in \cite[Section 3.2.3]{boyd2004convex}, the Lagrange dual function $\lagrangeDual$ is also convex.
\end{proof}

\begin{proof}[Proof of Theorem \ref{thm:strongDual}]
First, we prove the first statement. For each $\xVar$ in $\setXVar$, the maximal value of (\ref{eq:DROC}) is equal to the infimum of (\ref{eq:dualDROC}), based on Lemma \ref{lem:lagrangeDual}. Hence, the minimization regarding $\xVar$ in (\ref{eq:DROC}) and (\ref{eq:dualDROC}) are identical, yielding the first statement. %as follows:

%\begin{multline*}
%\min_{\xVar\in\setXVar}\quad\innerMax = \\ \min_{\xVar\in\setXVar}\quad\innerMin. 
%\end{multline*} 

Subsequently, the second statement is proven. %From Lemma \ref{lem:lagrangeDual}, the objective function of (\ref{eq:DROC}) is convex in $(\xVar,\ineqVars,\eqVar)\in\setXVar\times\realPlus^\dimID\times\real$. Therefore, the problem in (\ref{eq:DROC}) is convex programming. Hence, the problem has unique optimal solution if there exists any locally optimal solution.
Suppose that $\targetFcn$ is strictly convex and continuous, and $\setXVar$ is bounded and closed convex. By naturally extending the result in \cite[Section 3.2.3]{boyd2004convex} to strictly convex functions, the objective function of (\ref{eq:DROC}) is also strictly convex on $\setXVar$. Therefore, the set of optimal minimizers for the problem contains one point at most \cite[Section 4.2.1]{boyd2004convex}. In addition, according to the extreme value theorem \cite{keisler2012elementary}, the optimal minimizer set contains at least one point; therefore, it must be unique.   

Finally, we prove the third statement. Suppose that $\targetFcn$ is convex and class $C^k$ on an open convex set $\setXVar$. Then, based on Lemma \ref{lem:lagrangeDual}, 
the objective function $\innerDual$ is convex on $\setXVar\times\realPlus^\dimID\times\real$. In addition, the function $\dual$ is clearly class $C^k$ because $\targetFcn$ is class $C^k$ on $\setXVar\times\realPlusPlus^\dimID\times\real$. Hence, the third statement is proven.   
\end{proof}

\begin{proof}[Proof of Theorem \ref{thm:knapsack}]
Let $(1+\const)\,[\zVar]_\ID=\densRatio$; then, $\ambiguity=\ambiguityKnapsack$. Therefore, we obtain the following:

\begin{equation*}
\innerMax=\max_{\zVar\in\ambiguityKnapsack} \quad (1+\const)\innerKnapsack. 
\end{equation*} 
Hence, we obtain the first statement. 

In addition, let $\constC$ be in $\{1,\cdots,\dimID\}$, and: 

\begin{equation*}
(1+\const)\,\refProb=\frac{1}{\constC}. 
\end{equation*}
Then, the following equation is obtained: 

\begin{equation*}
\ambiguityKnapsack=\{\zVar\in[0,\,1]^\dimID~|~\frac{1}{\constC}\sum_{\ID\in\setID}[\zVar]_\ID=1\}.
\end{equation*}
Furthermore, the inner maximization problem of (\ref{eq:knapsackRC}) is convex maximization in $\zVar\in\ambiguityKnapsack$. Therefore, based on the result in \cite[Theorem 6.12]{6282663}, we only need to consider that $\zVar$ is in the vertices of $\ambiguityKnapsack$ and $\constC$ is in $\{1,\,\cdots,\,\dimID\}$, $\ambiguityKnapsack\cap\{0,\,1\}^\dimID$. 
Hence, we have the following results:

\begin{multline*}
\exists\zVar^*\in\ambiguityKnapsack\cap\{0,\,1\}^\dimID, \quad 
(1+\const)\sum_{\ID\in\setID}\refProb\,\targetFcnHead(\xVar,\,\ID)\,[\zVar^*]_\ID \\ 
=\innerMaxKnapsackUniform, 
\end{multline*}

\begin{multline*}
\exists(\ID_1^{*},\cdots,\ID_\constC^{*})\in\ambiguityC, \quad \\
\max_{\zVar\in\ambiguityKnapsack\cap\{0,\,1\}^\dimID} \quad (1+\const)\innerKnapsack=\sum_{l=1}^{\constC} \frac{\targetFcnHead(\xVar,\,\ID_l^{*})}{\constC}.
\end{multline*}
The last results are equivalent to the following inequalities: 

\begin{multline*}
\max_{\zVar\in\ambiguityKnapsack\cap\{0,1\}^\dimID} \quad (1+\const)\innerKnapsack \\
\ge\innerMaxKnapsackUniform,
\end{multline*}

\begin{multline*}
\max_{\zVar\in\ambiguityKnapsack\cap\{0,1\}^\dimID} \quad (1+\const)\innerKnapsack \\
\le\innerMaxKnapsackUniform. 
\end{multline*}
Hence, the second statement is obtained.
\end{proof}

\section{Numerical Experiments}

This section presents numerical examples to demonstrate the effectiveness of the proposed method. We compare the proposed DDROC method $(\const > 0)$ with the SOC method in terms of the worst-case and average performances.

\subsection{Settings: Patroller Agent Design}

Let us consider a finite undirected graph $\mathcal{G}(\setID,\,\mathcal{E})$ and a discrete-time Markov chain, which represents a sequence of a patroller agent state $X_t\in\setID$ for discrete time $t\in\{0,\,1,\,\cdots,\,\infty\}$. Here, $\setID=\{1,\,2,\,\cdots,\,\dimID\}$ is a set of nodes (states) and $\mathcal{E}\subseteq\setID\times\setID$ is a set of edges (connections between states). The Markov property is satisfied; namely, $\distribution[X_t=j_t~|~X_{t-1}=j_t-1]=\distribution[X_t=j_t~|~X_{t-1}=j_{t-1},\, \cdots,\, X_{0}=j_{0}]$ for all $j_t\in\setID$. 
The Markov chain has an associated transition matrix $\boldsymbol{P}\in\real^{\dimID\times\dimID}$, whose component in the $j$-th row and $k$-th column denotes the transition probability from state $j$ to $k$, $[\boldsymbol{P}]_{j,\,k}=\distribution[X_t=k~|~X_{t-1}=j]$. 

We introduce the mean hitting time minimization problem of a patrolling agent for a given graph \cite{diaz2023distributed}. For a Markov chain, the mean hitting time is defined as the average time to first reach goal states \cite{norris1998markov}. First, let us suppose that $\boldsymbol{P}$ belongs to $\mathcal{M}_{\boldsymbol{\pi}}^*$. Here, $\mathcal{M}_{\boldsymbol{\pi}}^*$ denotes the set of irreducible and reversible stochastic matrices with the stationary distribution of the Markov chain, $\boldsymbol{\pi}\in\realPlus^\dimID$; $[\boldsymbol{\pi}]_j$ is also given and matches the average time that the patroller spent for a state $j\in\setID$ in the long run. Subsequently, as defined in \cite{diaz2023distributed}, let us denote the mean hitting time to the set of goal states $\mathcal{A}(\ID)=\{\ID\}\subseteq\setID$ with a random variable $\ID$:

\begin{equation}\label{eq:hittingTime}
\targetFcn=\boldsymbol{\pi}^\top(\eye_{\dimID} - \boldsymbol{E}_{\mathcal{A}}(\ID) \boldsymbol{P} \boldsymbol{E}_{\mathcal{A}}(\ID))\boldsymbol{\delta}_{\mathcal{A}}(\ID). 
\end{equation}
Here, $\xVar=\vect(\boldsymbol{P})$ is a decision variable and $\boldsymbol{\delta}_{\mathcal{A}}(\ID)\in\real^\dimID$ is a vector valued in $\{0,1\}$. $[\boldsymbol{\delta}_{\mathcal{A}}(\ID)]_j=1$ if $j\notin\mathcal{A}(\ID)$; otherwise, $[\boldsymbol{\delta}_{\mathcal{A}}(\ID)]_j=0$. Futhermore, $\boldsymbol{E}_{\mathcal{A}}(\ID)=\diag(\boldsymbol{\delta}_{\mathcal{A}}(\ID))\in\real^{\dimID\times\dimID}$. The mean hitting time $\targetFcn$ in (\ref{eq:hittingTime}) is a convex function in $\xVar$ on $\mathcal{M}_{\boldsymbol{\pi}}^*$.

\subsection{DDROC problems of Patroller Agent Design}

We consider the DDROC problem in (\ref{eq:DROC}) associated with the mean hitting time $\targetFcn$ in (\ref{eq:hittingTime}). Originally in (\ref{eq:hittingTime}), the objective function was either a weighted average value or the worst-case value of the mean hitting time. Unfortunately, these weights are often difficult to assign because noteworthy nodes may be unknown. Instead of assigning the weights, we consider the distributional uncertainty $\ambiguity$ of the weights $\prob$, as denoted in the following: 
%By using our proposed method, these weights no longer need to be assigned. The reformulated DDROC problem in (\ref{eq:dualDROC}) associated with (\ref{eq:hittingTime}) becomes a smooth convex programming function because $\targetFcn$ in (\ref{eq:hittingTime}) is a convex and clearly smooth function. The DDROC problem is denoted as follows:

\begin{equation}\label{eq:hittingTimeDROC}
\begin{split}
&\min_{\boldsymbol{P}\in\mathcal{M}_{\boldsymbol{\pi}}^*} \quad \max_{\probHead\in\ambiguity} \quad \expectation{\prob}{\boldsymbol{\pi}^\top(\eye_{\dimID} - \boldsymbol{E}_{\mathcal{A}}(\ID) \boldsymbol{P} \boldsymbol{E}_{\mathcal{A}}(\ID) \boldsymbol{\delta}_{\mathcal{A}}(\ID)}, \\
&\text{s.t.}~[\boldsymbol{P}]_{j,\,k}=0,\quad \forall (j,\,k)\notin\mathcal{E}.
\end{split}
\end{equation}
Theorem \ref{thm:knapsack} states that the problem in (\ref{eq:hittingTimeDROC}) is equivalent to the minimization of the average $\targetFcn$ of the worst $\constC$ nodes. Specifically, we aim to set $\ambiguity$, or in short, set $\refProb$ to the uniform distribution and set $\const$ such that $\constC$ matches the desired value. Furthermore, Theorem \ref{thm:strongDual} states that the reformulated DDROC problem in (\ref{eq:dualDROC}) reduces to a smooth convex programming problem because $\targetFcn$ in (\ref{eq:hittingTime}) is convex and clearly smooth. 

We consider two types of graphs in numerical experiments: 1) the San Francisco data set in \cite{diaz2023distributed} ($|\setID|=\dimID=18$ and $|\mathcal{E}|=78$), and 2) the random graphs generated using the Watts-Strogatz model\cite{watts1998collective}. Fig. \ref{figurelabel} shows the graph of the San Francisco data set. The sizes of the random graphs were set to $|\setID|=\dimID=36$ or $54$. The average degree of these random graphs, including the edges of the self-loop, was set to three. The randomness parameter of the Watts-Strogatz model was set to $0.05$. The code in \cite{mathworksBuildWattsStrogatz} was used to create the random graphs, which were created five times for each size, using different random seeds. The stationary distribution $\boldsymbol{\pi}$ for each graph was set to the uniform distribution. 

For the DDROC problem in (\ref{eq:hittingTimeDROC}), $\const$ in (\ref{eq:ambiguity}) which defines the size of $\ambiguity$, was set so that $\constC=\dimID/\,(1+\const)$ is integer. %The nominal probability $\refProb$ was set to the uniform distribution. 

\begin{figure}[t]
   \centering
   %\framebox{\parbox{3in}{We suggest that you use a text box to insert a graphic (which is ideally a 300 dpi TIFF or EPS file, with all fonts embedded) because, in an document, this method is somewhat more stable than directly inserting a picture.}}
   \includegraphics[scale=0.65,clip]{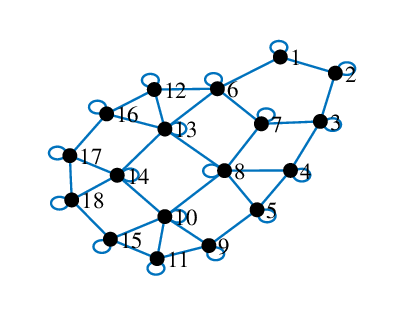}
   \caption{Graph of the San Francisco data set in \cite{diaz2023distributed}.}
   \label{figurelabel}
\end{figure}

\subsection{Verification of Solvability and the Explainability}

We used the fmincon function in MATLAB\cite{coleman1999optimization} to solve the SOC and DDROC problems. Table \ref{table:result} lists the mean hitting time results for each local optimal solution determined by the fmincon function. 

A global optimal solution for general smooth convex programming can be obtained using solvers that find a local optimal solution (solvability). In addition, as described in Theorem \ref{thm:knapsack}, we can confirm that the proposed method with each $\constC$ could minimize the average value of the mean hitting time collection of the worst $\constC$ nodes in Table \ref{table:result} (explainability). 

\begin{table}[t]
\caption{Results of Numerical Experiments (Average Value of the Mean Hitting Time for Each Node). }
\label{table:result}
\begin{center}
\begin{tabular}{l||c|c|c||c}
\hline

\multicolumn{5}{l}{San Francisco ($|\setID|\,(=\dimID)=18,~|\mathcal{E}|=78$)}  \\
\hline
\multirow{2}{*}{Node} & \multicolumn{3}{c||}{Proposed Method} & SOC \\ \cline{2-4}
    & $\constC=1$ & $\constC=5$ & $\constC=9$ & Method \\

\hline
Worst & $\boldsymbol{37.4}$ & $38.7$ & $39.8$ & $41.1$ \\
\hline
Worst $5$ & $35.7$ & $\boldsymbol{33.9}$ & $34.1$ & $34.7$ \\
\hline
Worst $9$ & $33.9$ & $32.4$ & $\boldsymbol{32.3}$ & $32.5$ \\
\hline
Mean & $30.6$ & $29.4$ & $29.1$ & $\boldsymbol{29.0}$ \\
\hline \hline

%\multicolumn{5}{l}{Watts-Strogatz Model ($|\setID|\,(=\dimID)=18, |\epsilon|=18\times3$)}  \\
%\hline
%\multirow{2}{*}{Node} & \multicolumn{3}{c||}{Proposed Method} & SOC \\ \cline{2-4}
%    & $\constC=1$ & $\constC=5$ & $\constC=9$ & %Method \\

%\hline
%Worst & $\boldsymbol{27.3}$ & $27.7$ & $28.3$ & $29.3$ \\
%\hline
%Worst $5$ & $27.3$ & $\boldsymbol{27.2}$ & $27.5$ & $28.0$ \\
%\hline
%Worst $9$ & $27.3$ & $27.2$ & $\boldsymbol{26.9}$ & $27.3$ \\
%\hline
%Mean & $26.8$ & $26.6$ & $26.2$ & $\boldsymbol{25.8}$ \\
%\hline \hline

\multicolumn{5}{l}{Watts-Strogatz Model ($|\setID|\,(=\dimID)=36,~|\mathcal{E}|=36\times3$)}  \\
\hline
\multirow{2}{*}{Node} & \multicolumn{3}{c||}{Proposed Method} & SOC \\ \cline{2-4}
    & $\constC=1$ & $\constC=9$ & $\constC=18$ & Method \\

\hline
Worst & $\boldsymbol{84.0}$ & $86.2$ & $87.4$ & $88.8$ \\
\hline
Worst $9$ & $82.0$ & $\boldsymbol{81.3}$ & $82.0$ & $83.2$ \\
\hline
Worst $18$ & $79.5$ & $\boldsymbol{78.5}$ & $\boldsymbol{78.5}$ & $79.2$ \\
\hline
Mean & $74.1$ & $73.3$ & $72.8$ & $\boldsymbol{72.4}$ \\
\hline \hline

\multicolumn{5}{l}{Watts-Strogatz Model ($|\setID|\,(=\dimID)=54,~|\mathcal{E}|=54\times3$)}  \\
\hline
\multirow{2}{*}{Node} & \multicolumn{3}{c||}{Proposed Method} & SOC \\ \cline{2-4}
    & $\constC=1$ & $\constC=18$ & $\constC=27$ & Method \\

\hline
Worst & $\boldsymbol{162.9}$ & $167.9$ & $170.5$ & $173.0$ \\
\hline
Worst $18$ & $159.7$ & $\boldsymbol{157.9}$ & $159.3$ & $161.2$ \\
\hline
Worst $27$ & $155.1$ & $153.1$ & $\boldsymbol{152.7}$ & $154.1$ \\
\hline
Mean & $142.9$ & $141.4$ & $140.1$ & $\boldsymbol{139.3}$ \\
\hline
\end{tabular}
\end{center}
\end{table}

\section{Conclusions}
This study presents a DDROC reformulation method based on the density ratio, which bounds the TV distance\cite{10.3150/20-BEJ1256}. The proposed method reformulates DDROC problems into one-layer smooth convex programming problems with only a non-negative constraint of the Lagrange multiplier. Demonstrations associated with the patroller-agent design were performed. 

Specifically, in this study, we studied DDROC problems without constraints related to distributional uncertainties. Problems involving distributionally robust constraints merit further study. Another challenge is to extend the proposed theory to a superset of optimal transport balls.

\addtolength{\textheight}{-0cm}   % This command serves to balance the column lengths

%%%%%%%%%%%%%%%%%%%%%%%%%%%%%%%%%%%%%%%%%%%%%%%%%%%%%%%%%%%%%%%%%%%%%%%%%%%%%%%%

%%%%%%%%%%%%%%%%%%%%%%%%%%%%%%%%%%%%%%%%%%%%%%%%%%%%%%%%%%%%%%%%%%%%%%%%%%%%%%%%

%%%%%%%%%%%%%%%%%%%%%%%%%%%%%%%%%%%%%%%%%%%%%%%%%%%%%%%%%%%%%%%%%%%%%%%%%%%%%%%%
%\section*{APPENDIX}

%\section*{ACKNOWLEDGMENT}

%The preferred spelling of the word �acknowledgment� in America is without an �e� after the �g�. Avoid the stilted expression, �One of us (R. B. G.) thanks . . .�  Instead, try �R. B. G. thanks�. Put sponsor acknowledgments in the unnumbered footnote on the first page.

%%%%%%%%%%%%%%%%%%%%%%%%%%%%%%%%%%%%%%%%%%%%%%%%%%%%%%%%%%%%%%%%%%%%%%%%%%%%%%%%

%References are important to the reader; therefore, each citation must be complete and correct. If at all possible, references should be commonly available publications.

\bibliography{AACC2025}

\begin{thebibliography}{10}
\providecommand{\url}[1]{#1}
\csname url@rmstyle\endcsname
\providecommand{\newblock}{\relax}
\providecommand{\bibinfo}[2]{#2}
\providecommand\BIBentrySTDinterwordspacing{\spaceskip=0pt\relax}
\providecommand\BIBentryALTinterwordstretchfactor{4}
\providecommand\BIBentryALTinterwordspacing{\spaceskip=\fontdimen2\font plus
\BIBentryALTinterwordstretchfactor\fontdimen3\font minus
  \fontdimen4\font\relax}
\providecommand\BIBforeignlanguage[2]{{%
\expandafter\ifx\csname l@#1\endcsname\relax
\typeout{** WARNING: IEEEtran.bst: No hyphenation pattern has been}%
\typeout{** loaded for the language `#1'. Using the pattern for}%
\typeout{** the default language instead.}%
\else
\language=\csname l@#1\endcsname
\fi
#2}}

\bibitem{duan2021markov}
X.~Duan and F.~Bullo, ``Markov chain--based stochastic strategies for robotic
  surveillance,'' \emph{Annual Review of Control, Robotics, and Autonomous
  Systems}, vol.~4, no.~1, pp. 243--264, 2021.

\bibitem{diaz2023distributed}
G.~D{\'\i}az-Garc{\'\i}a, F.~Bullo, and J.~R. Marden, ``Distributed markov
  chain-based strategies for multi-agent robotic surveillance,'' \emph{IEEE
  Control Systems Letters}, vol.~7, pp. 2527--2532, 2023.

\bibitem{bertsekas1996stochastic}
D.~Bertsekas and S.~E. Shreve, \emph{Stochastic optimal control: the
  discrete-time case}.\hskip 1em plus 0.5em minus 0.4em\relax Athena
  Scientific, 1996, vol.~5.

\bibitem{crespo2003stochastic}
L.~G. Crespo and J.-Q. Sun, ``Stochastic optimal control via bellman's
  principle,'' \emph{Automatica}, vol.~39, no.~12, pp. 2109--2114, 2003.

\bibitem{sastry2011adaptive}
S.~Sastry and M.~Bodson, \emph{Adaptive control: stability, convergence and
  robustness}.\hskip 1em plus 0.5em minus 0.4em\relax Courier Corporation,
  2011.

\bibitem{slotine1985robust}
J.-J.~E. Slotine, ``The robust control of robot manipulators,'' \emph{The
  International Journal of Robotics Research}, vol.~4, no.~2, pp. 49--64, 1985.

\bibitem{scherer2001theory}
C.~Scherer, ``Theory of robust control,'' \emph{Delft University of
  Technology}, pp. 1--160, 2001.

\bibitem{nishimura2021rat}
H.~Nishimura, N.~Mehr, A.~Gaidon, and M.~Schwager, ``Rat ilqr: A risk
  auto-tuning controller to optimally account for stochastic model mismatch,''
  \emph{IEEE Robotics and Automation Letters}, vol.~6, no.~2, pp. 763--770,
  2021.

\bibitem{taskesen2024distributionally}
B.~Taskesen, D.~Iancu, {\c{C}}.~Ko{\c{c}}yi{\u{g}}it, and D.~Kuhn,
  ``Distributionally robust linear quadratic control,'' \emph{Advances in
  Neural Information Processing Systems}, vol.~36, 2024.

\bibitem{liu2023data}
R.~Liu, G.~Shi, and P.~Tokekar, ``Data-driven distributionally robust optimal
  control with state-dependent noise,'' in \emph{2023 IEEE/RSJ International
  Conference on Intelligent Robots and Systems (IROS)}.\hskip 1em plus 0.5em
  minus 0.4em\relax IEEE, 2023, pp. 9986--9991.

\bibitem{yang2020wasserstein}
I.~Yang, ``Wasserstein distributionally robust stochastic control: A
  data-driven approach,'' \emph{IEEE Transactions on Automatic Control},
  vol.~66, no.~8, pp. 3863--3870, 2020.

\bibitem{nguyen2023distributionally}
H.~T. Nguyen and D.-H. Choi, ``Distributionally robust model predictive control
  for smart electric vehicle charging station with v2g/v2v capability,''
  \emph{IEEE Transactions on Smart Grid}, vol.~14, no.~6, pp. 4621--4633, 2023.

\bibitem{coulson2021distributionally}
J.~Coulson, J.~Lygeros, and F.~D{\"o}rfler, ``Distributionally robust chance
  constrained data-enabled predictive control,'' \emph{IEEE Transactions on
  Automatic Control}, vol.~67, no.~7, pp. 3289--3304, 2021.

\bibitem{van2015distributionally}
B.~P. Van~Parys, D.~Kuhn, P.~J. Goulart, and M.~Morari, ``Distributionally
  robust control of constrained stochastic systems,'' \emph{IEEE Transactions
  on Automatic Control}, vol.~61, no.~2, pp. 430--442, 2015.

\bibitem{pilipovsky2024distributionally}
J.~Pilipovsky and P.~Tsiotras, ``Distributionally robust density control with
  wasserstein ambiguity sets,'' \emph{arXiv preprint arXiv:2403.12378}, 2024.

\bibitem{hu2013kullback}
Z.~Hu and L.~J. Hong, ``Kullback-leibler divergence constrained
  distributionally robust optimization,'' \emph{Available at Optimization
  Online}, vol.~1, no.~2, p.~9, 2013.

\bibitem{gao2023distributionally}
R.~Gao and A.~Kleywegt, ``Distributionally robust stochastic optimization with
  wasserstein distance,'' \emph{Mathematics of Operations Research}, vol.~48,
  no.~2, pp. 603--655, 2023.

\bibitem{shafieezadeh2019regularization}
S.~Shafieezadeh-Abadeh, D.~Kuhn, and P.~M. Esfahani, ``Regularization via mass
  transportation,'' \emph{Journal of Machine Learning Research}, vol.~20, no.
  103, pp. 1--68, 2019.

\bibitem{cherukuri2022data}
A.~Cherukuri, A.~Zolanvari, G.~Banjac, and A.~R. Hota, ``Data-driven
  distributionally robust optimization over a network via distributed
  semi-infinite programming,'' in \emph{2022 IEEE 61st Conference on Decision
  and Control (CDC)}.\hskip 1em plus 0.5em minus 0.4em\relax IEEE, 2022, pp.
  4771--4775.

\bibitem{luo2017decomposition}
F.~Luo and S.~Mehrotra, ``Decomposition algorithm for distributionally robust
  optimization using wasserstein metric,'' \emph{arXiv preprint
  arXiv:1704.03920}, 2017.

\bibitem{mohajerin2018data}
P.~Mohajerin~Esfahani and D.~Kuhn, ``Data-driven distributionally robust
  optimization using the wasserstein metric: Performance guarantees and
  tractable reformulations,'' \emph{Mathematical Programming}, vol. 171, no.~1,
  pp. 115--166, 2018.

\bibitem{reemtsen2013semi}
R.~Reemtsen and J.-J. R{\"u}ckmann, \emph{Semi-infinite programming}.\hskip 1em
  plus 0.5em minus 0.4em\relax Springer Science \& Business Media, 2013,
  vol.~25.

\bibitem{mehrotra2014cutting}
S.~Mehrotra and D.~Papp, ``A cutting surface algorithm for semi-infinite convex
  programming with an application to moment robust optimization,'' \emph{SIAM
  Journal on Optimization}, vol.~24, no.~4, pp. 1670--1697, 2014.

\bibitem{wiesemann2014distributionally}
W.~Wiesemann, D.~Kuhn, and M.~Sim, ``Distributionally robust convex
  optimization,'' \emph{Operations research}, vol.~62, no.~6, pp. 1358--1376,
  2014.

\bibitem{miao2021data}
F.~Miao, S.~He, L.~Pepin, S.~Han, A.~Hendawi, M.~E. Khalefa, J.~A. Stankovic,
  and G.~Pappas, ``Data-driven distributionally robust optimization for vehicle
  balancing of mobility-on-demand systems,'' \emph{ACM Transactions on
  Cyber-Physical Systems}, vol.~5, no.~2, pp. 1--27, 2021.

\bibitem{staib2019distributionally}
M.~Staib and S.~Jegelka, ``Distributionally robust optimization and
  generalization in kernel methods,'' \emph{Advances in Neural Information
  Processing Systems}, vol.~32, 2019.

\bibitem{liu2019discrete}
Y.~Liu, A.~Pichler, and H.~Xu, ``Discrete approximation and quantification in
  distributionally robust optimization,'' \emph{Mathematics of Operations
  Research}, vol.~44, no.~1, pp. 19--37, 2019.

\bibitem{farokhi2023distributionally}
F.~Farokhi, ``Distributionally robust optimization with noisy data for discrete
  uncertainties using total variation distance,'' \emph{IEEE Control Systems
  Letters}, vol.~7, pp. 1494--1499, 2023.

\bibitem{zhang2021efficient}
Z.~Zhang, S.~Ahmed, and G.~Lan, ``Efficient algorithms for distributionally
  robust stochastic optimization with discrete scenario support,'' \emph{SIAM
  Journal on Optimization}, vol.~31, no.~3, pp. 1690--1721, 2021.

\bibitem{villani2009optimal}
C.~Villani \emph{et~al.}, \emph{Optimal transport: old and new}.\hskip 1em plus
  0.5em minus 0.4em\relax Springer, 2009, vol. 338.

\bibitem{10.3150/20-BEJ1256}
\BIBentryALTinterwordspacing
L.~D{\"u}mbgen, R.~J. Samworth, and J.~A. Wellner, ``{Bounding distributional
  errors via density ratios},'' \emph{Bernoulli}, vol.~27, no.~2, pp. 818 --
  852, 2021. [Online]. Available: \url{https://doi.org/10.3150/20-BEJ1256}
\BIBentrySTDinterwordspacing

\bibitem{boyd2004convex}
S.~Boyd and L.~Vandenberghe, \emph{Convex optimization}.\hskip 1em plus 0.5em
  minus 0.4em\relax Cambridge university press, 2004.

\bibitem{cover1999elements}
T.~M. Cover, \emph{Elements of information theory}.\hskip 1em plus 0.5em minus
  0.4em\relax John Wiley \& Sons, 1999.

\bibitem{keisler2012elementary}
H.~J. Keisler, \emph{Elementary calculus: An infinitesimal approach}.\hskip 1em
  plus 0.5em minus 0.4em\relax Courier Corporation, 2012.

\bibitem{6282663}
B.~Schölkopf and A.~J. Smola, \emph{Optimization}, 2001, pp. 149--186.

\bibitem{norris1998markov}
J.~R. Norris, \emph{Markov chains}.\hskip 1em plus 0.5em minus 0.4em\relax
  Cambridge university press, 1998, no.~2.

\bibitem{watts1998collective}
D.~J. Watts and S.~H. Strogatz, ``Collective dynamics of
  ‘small-world’networks,'' \emph{nature}, vol. 393, no. 6684, pp. 440--442,
  1998.

\bibitem{mathworksBuildWattsStrogatz}
``{B}uild {W}atts-{S}trogatz {S}mall {W}orld {G}raph {M}odel -
  {M}{A}{T}{L}{A}{B} \&amp; {S}imulink {E}xample --- mathworks.com,''
  \url{https://www.mathworks.com/help/matlab/math/build-watts-strogatz-small-world-graph-model.html},
  [Accessed 07-08-2024].

\bibitem{coleman1999optimization}
T.~Coleman, M.~A. Branch, and A.~Grace, ``Optimization toolbox,'' \emph{For Use
  with MATLAB. User’s Guide for MATLAB 5, Version 2, Relaese II}, 1999.

\end{thebibliography}
\bibliographystyle{IEEEtran}

\end{document}